\pdfoutput=1
\RequirePackage{ifpdf}
\ifpdf 
\documentclass[pdftex]{sigma}
\else
\documentclass{sigma}
\fi

\usepackage[shortlabels]{enumitem}

\usepackage{textgreek}
\usepackage{mleftright}
\mleftright

\numberwithin{equation}{section}

\newtheorem{Theorem}{Theorem}[section]
\newtheorem*{Theorem*}{Theorem}

\newtheorem{Lemma}[Theorem]{Lemma}

 { \theoremstyle{definition}
\newtheorem{Definition}[Theorem]{Definition}

\newtheorem{Remark}[Theorem]{Remark} }

\newcommand{\C}{\mathbb{C}}

\newcommand{\dd}{\mathrm{d}}

\renewcommand{\sp}{\hspace{0.5pt}}

\setlist[enumerate]{label=(\roman*)}

\begin{document}

\newcommand{\arXivNumber}{2506.?????}

\renewcommand{\thefootnote}{}

\renewcommand{\PaperNumber}{051}

\FirstPageHeading

\ShortArticleName{Factorization of Basic Hypergeometric Series}

\ArticleName{Factorization of Basic Hypergeometric Series\footnote{This paper is a~contribution to the Special Issue on Basic Hypergeometric Series Associated with Root Systems and Applications in honor of Stephen C.~Milne's 75th birthday. The~full collection is available at \href{https://www.emis.de/journals/SIGMA/Milne.html}{https://www.emis.de/journals/SIGMA/Milne.html}}}

\Author{Jonathan G.~BRADLEY-THRUSH}

\AuthorNameForHeading{J.G.~Bradley-Thrush}

\Address{Grupo de F\'{\i}sica Matem\'atica, Instituto Superior T\'ecnico, Universidade de Lisboa, \\ Av. Rovisco Pais, 1049-001 Lisboa, Portugal}
\Email{\href{mailto:jgbradleythrush@ciencias.ulisboa.pt}{jgbradleythrush@ciencias.ulisboa.pt}}

\ArticleDates{Received November 30, 2024, in final form June 22, 2025; Published online July 06, 2025}

\Abstract{The general problem of the factorization of a basic hypergeometric series is presented and discussed. The case of the general $_2\psi_2$ series is examined in detail. Connections are found with the theory of basic hypergeometric series on root systems. Alternative proofs of several well-known summation and transformation formulae, including Gustafson's generalization of Ramanujan's $_1\psi_1$ summation, are obtained incidentally.}

\Keywords{basic hypergeometric series; theta functions; elliptic functions; $q$-difference equations}

\Classification{33D15; 33E05; 33D67; 39A13}

\renewcommand{\thefootnote}{\arabic{footnote}}
\setcounter{footnote}{0}

\section{Introduction}

Throughout this paper, $q$ is taken to denote a complex number satisfying $0<|q|<1$. A basic hypergeometric series is any sum $\sum_{n=M}^N u_n$ for which the ratio $u_{n+1}/u_n$ is a rational function of~$q^n$. Here, $M$ and $N$ may take any integer values subject to the condition that~${M \leq N}$, and one may also take $M=-\infty$, or $N=\infty$, or both. It is conventional to write the summand~$u_n$ in terms of a~ratio of products of finite $q$-Pochhammer symbols, $(x)_n$, defined such as to take the value~$1$ when~${n=0}$ and extended to all integer subscripts $n$ by the recursion~${(x)_{n+1} = (1-xq^n)(x)_n}$. The infinite $q$-Pochhammer symbol, $(x)_\infty$, is the infinite product given by the limit $\lim_{n \rightarrow \infty} (x)_n$. The abbreviation $(x_1,x_2, \dots, x_m)_n$ is used for the product~${(x_1)_n (x_2)_n \cdots (x_m)_n}$. The notation~${\theta(x) = (x,q/x,q)_\infty}$ is used for the Jacobian theta function.

The question of whether a given basic hypergeometric series may be written explicitly as an infinite product has a long history. In the 1840s, both Jacobi~\cite[equation~(15)]{Jacobi1846} and Heine~\cite[equation~(80)]{Heine1847} obtained the summation
\begin{equation} \label{eqn:q-Gauss}
\sum_{n=0}^\infty {(a,b)_n \over (q,abx)_n} x^n = {(ax,bx)_\infty \over (x,abx)_\infty},
\end{equation}
valid for $|x|<1$, which is a $q$-analogue of Gauss's $_2F_1$ summation. The special case ${b=0}$ of this formula had been found already by Cauchy~\cite[Section~I, equation~(15)]{Cauchy1843}, and the case~${a=b=0}$ goes back to Euler~\cite[Section~25]{Euler1753}. Other well-known factorizations include the bilateral summations
\begin{equation} \label{eqn:1psi1}
\sum_{n=-\infty}^\infty {(a)_n \over (b)_n} x^n = {\bigl({b \over a}\bigr)_\infty \theta(ax) \over \bigl(x,{b \over ax},b,{q \over a}\bigr)_\infty},
\end{equation}
valid for $|b/a|<|x|<1$, and
\begin{gather}
\sum_{n=-\infty}^\infty {\bigl(a_1y,a_2y,a_3y,a_4y\bigr)_n \over \bigl({qy \over a_1}, {qy \over a_2}, {qy \over a_3}, {qy \over a_4}\bigr)_n} \bigl( 1 - y^2q^{2n} \bigr) \bigg( {q \over a_1a_2a_3a_4} \bigg)^{ n}\nonumber \\
\qquad = {\bigl({q \over a_1a_2}, {q \over a_1a_3}, {q \over a_1a_4}, {q \over a_2a_3}, {q \over a_2a_4}, {q \over a_3a_4}\bigr)_\infty \theta\bigl(y^2\bigr) \over \bigl({q \over a_1y}, {q \over a_2y}, {q \over a_3y}, {q \over a_4y}, {qy \over a_1}, {qy \over a_2}, {qy \over a_3}, {qy \over a_4}, {q \over a_1a_2a_3a_4}\bigr)_\infty},\label{eqn:6psi6}
\end{gather}
valid for $|a_1a_2a_3a_4|>|q|$, due respectively to Ramanujan (see~\cite[Sections~3.2 and~3.3]{Andrews2009} or~\cite[pp.~222--223]{Hardy1940}) and Bailey~\cite[equation~(4.7)]{Bailey1936}. There are also a number of series for which it is generally accepted that, for general values of the parameters involved, no summation formula exists. An example of such a series is
\begin{equation} \label{eqn:2phi1_series}
\sum_{n=0}^\infty {(a,b)_n \over (q,c)_n} x^n,
\end{equation}
which contains the general $_2F_1$ series as a limiting case. In the case of this particular series, the non-existence of a summation formula is stated explicitly by Johnson~\cite[p.~227]{Johnson2020}. Another prominent example is the partial theta series,
\begin{equation} \label{eqn:partial_theta_series}
\sum_{n=0}^\infty (-1)^n q^{n(n-1) \over 2} x^n.
\end{equation}
Being an entire function of $x$, this does of course admit a Weierstrass product representation; several of its factors were recorded by Ramanujan in his notebooks~\cite[p.~26]{Ramanujan1988} (see~\cite[pp.~285--286]{Andrews2005}). Likewise~\eqref{eqn:2phi1_series}, continued analytically beyond the disk $|x|<1$, represents a meromorphic function of $x$ on the whole complex plane. Its poles are known to occur at the points $x=q^{-n}$, $n \geq 0$, and hence we may write~\eqref{eqn:2phi1_series} in the form\footnote{The factorization of the numerator here is an application of Hadamard's theorem. The absence of exponential factors is due to the fact that $(x)_\infty \sum_{n=0}^\infty (a,b)_n x^n/(q,c)_n$ is entire of order $0$. This can be deduced from~\cite[equation~(4.3.2)]{Gasper2004}, but it may also be proved as follows: This function admits a power series expansion $\sum_{n=0}^\infty S_n x^n$ in which the coefficients are the finite sums
\[ S_n = \sum_{m=0}^n {(-1)^{n-m} q^{(n-m)(n-m-1) \over 2} (a,b)_m \over (q,c)_m (q)_{n-m}} = {(-1)^n q^{n(n-1) \over 2} a^n \over (c)_n} \sum_{m=0}^n {(a)_m (c/a)_{n-m} \over (q)_m (q)_{n-m}} \bigg( {b \over a} \bigg)^{ m}. \]
The first equality here follows easily from the power series expansion of $(x)_\infty$, while the second is a special case of one of Sears's transformation formulae (see~\cite[equation~(3.2.2)]{Gasper2004}). From the second expression for $S_n$, it is apparent that $1/\log|S_n| = O\bigl(1/n^2\bigr)$, which implies that the order of $\sum_{n=0}^\infty S_n x^n$ equals zero.} $(x)_\infty^{-1}\prod_{n=1}^\infty (1-x/\xi_n)$, where the values of the zeros $\xi_n$ depend on the parameters $a$, $b$, $c$ and $q$. There is, however, no general formula known for the $n$th zero $\xi_n$.

The purpose of this paper is to examine a different sense in which every basic hypergeometric series may be said to admit a factorization---one which involves only finitely many unknowns. There is no loss of generality in restricting our attention to bilateral series, since, given any finite or unilateral series, it is always possible to produce a bilateral series reducible to it by a suitable specialization of variables. We may, for example, obtain~\eqref{eqn:partial_theta_series} from the series~${\sum_{n=-\infty}^\infty (-1)^n q^{n(n-1)/2} x^n (a)_n/(b)_n}$ by first setting $b=q$ and then setting $a=q$. Further, there is no loss of generality in supposing that the number of $q$-Pochhammer symbols present in the numerator and denominator are equal, since this can be seen to incorporate the case of unequal numbers of factors by inserting extra factors of $(0)_n$ in the denominator of any such series.

The notation
\begin{equation} \label{eqn:rpsir_def}
_r\psi_r(x,y) := \sum_{n=-\infty}^\infty {(a_1 y, a_2 y, \dots, a_r y)_n \over (b_1 y, b_2 y, \dots, b_r y)_n} x^n
\end{equation}
will be used for a bilateral basic hypergeometric series, the parameters $a_j$ and $b_j$ being complex numbers independent of the variables $x$ and $y$ with the restriction that each $a_j$ is non-zero. Since the variable $y$ can be removed by rescaling the other parameters, there would be no loss of generality in setting it equal to $1$. It has nonetheless been inserted in the definition~\eqref{eqn:rpsir_def} because it plays a significant role in the methods of this paper: the goal will be to investigate the factorization of the function $_r\psi_r$ with respect to $y$.

The series in~\eqref{eqn:rpsir_def} converges when $x$ lies within the annulus $|b_1\cdots b_r/a_1\cdots a_r|<|x|<1$. Elsewhere, the function $_r\psi_r$ is defined by analytic continuation. When the value of $b_1\cdots b_r/a_1\cdots a_r$ is such that the annulus of convergence is empty, the series may instead be split into two unilateral series, one convergent for $|x|<1$ and the other convergent for $|x|>|b_1\cdots b_r/a_1\cdots a_r|$; separately each of the two series may then be continued beyond its region of convergence.\footnote{An alternative is to regard one of the parameters $a_j$ as variable, and to continue the series beyond its region of convergence as a function of $a_j$.} Analytic continuation of these unilateral series is discussed in~\cite[Section~4.5]{Gasper2004}. It is known that the function $_r\psi_r$ is meromorphic on $(\C \backslash \{0\})^2$, which is to say that it is expressible as the ratio of two functions analytic there. Since it is a straightforward matter to locate the poles of the function,\footnote{The poles at $y=q^n/a_j$ and at $y=b_j q^{-n}$ ($n \geq 0$ and $1 \leq j \leq r$) arise from singularities in the terms of the series in~\eqref{eqn:rpsir_def}. The poles at $x=q^{-n}$ and at $x = b_1 b_2 \cdots b_r q^n /a_1 a_2 \cdots a_r$ ($n \geq 0$) arise from analytic continuation of the series, as may be seen from~\cite[equation~(4.5.2)]{Gasper2004}.} the factorization of its denominator is known completely: we may write
\[ _r\psi_r(x,y) = {_r^{}\psi_r^\ast(x,y) \over \bigl(x,{b_1 b_2 \cdots b_r \over a_1 a_2 \cdots a_r x},{q \over a_1y},b_1y,{q \over a_2y},b_2y,\dots,{q \over a_ry},b_ry\bigr)_{ \infty}}, \]
where the numerator, $_r^{}\psi_r^\ast$, is analytic on $(\C \backslash \{0\})^2$. It is, incidentally, also analytic as a function of the parameters $a_j$ and $b_j$. In~\cite[Section~3.4]{Bradley-Thrush2023}, the following result on the factorization of~$_r^{}\psi_r^\ast$ with respect to $y$ is obtained; its proof is summarized briefly in Section~\ref{sect:theta} of this paper. This result is implicit in the work of Ito and Sanada. With respect to the variable $y$, they use essentially the same normalization of the $_r\psi_r$ function as is considered here.\footnote{In their terminology, it is the regularized Jackson integral of Jordan--Pochhammer type. (See~\cite[Sections~2.3 and~5.1]{Ito2008}.) Their version lacks the factors of $(x)_\infty$ and $(b_1 \cdots b_r/a_1 \cdots a_r x)_\infty$, but these are immaterial for the statement of Theorem~\ref{thm:general_rpsir_factorization}.} Although they do not state the factorization explicitly, it can be obtained by combining~\cite[Lemma~5.5]{Ito2008} with Lemma~\ref{lem:theta_factorization} of the present paper, which is a classical result.

\begin{Theorem} \label{thm:general_rpsir_factorization}
Let $r$ be a positive integer. There are functions $A$, $\rho_1, \rho_2, \dots, \rho_r$, independent of the variable $y$, such that the function $_r^{}\psi_r^\ast$ admits the factorization
\begin{equation} \label{eqn:general_rpsir_factorization}
_r^{}\psi_r^\ast(x,y) = A(x) \theta\bigg( {y \over \rho_1(x)} \bigg) \theta\bigg( {y \over \rho_2(x)} \bigg) \cdots \theta\bigg( {y \over \rho_r(x)} \bigg).
\end{equation}
Moreover, the functions $\rho_j$ are related by
\begin{equation} \label{eqn:rho_product}
\rho_1(x) \rho_2(x) \cdots \rho_r(x) = {1 \over a_1 a_2 \cdots a_r x}.
\end{equation}
\end{Theorem}

It is possible to produce a refinement of Theorem~\ref{thm:general_rpsir_factorization} which reduces the number of unknowns by one, expressing the function $A$ in terms of the functions $\rho_j$. For general values of $r$, this is sketched at the end of the paper. The case $r=2$ is treated in detail in Section~\ref{sect:2psi2_series}. (See Theorem~\ref{thm:2psi2_factorization} below.)

It is easy to check that Theorem~\ref{thm:general_rpsir_factorization} holds for certain special values of $x$ for which the factorization is elementary. One of them is $x=1$; from Abel's theorem \big(in the form which asserts $\lim_{x \rightarrow 1^-} (1-x)\sum_{n=0}^\infty u_n x^n = \lim_{n \rightarrow \infty} u_n$ whenever the second of these two limits exists\big), it follows that
\begin{equation} \label{eqn:rpsir_factorization_at_x=1}
_r\psi_r^\ast(1,y) = {\bigl({b_1\cdots b_r \over a_1\cdots a_r}\bigr)_\infty \over (q)_\infty^{r-1}} \theta(a_1y) \theta(a_2y) \cdots \theta(a_ry).
\end{equation}
Moreover, for any integer $k \geq 0$, the values taken by the functions $A$ and $\rho_j$ are readily determined at the points $x=q^{-k}$ and $x=q^kb_1\cdots b_r/a_1\cdots a_r$.

The functions introduced in Theorem~\ref{thm:general_rpsir_factorization} are known explicitly when a basic hypergeometric series is known to be expressible in closed form in terms of $q$-Pochhammer symbols. For example, Ramanujan's $_1\psi_1$ summation~\eqref{eqn:1psi1} asserts that
\begin{equation} \label{eqn:normalized_1psi1}
_1^{}\psi_1^\ast(x,y) = \left({b_1 \over a_1}\right)_{ \infty} \theta(a_1xy).
\end{equation}
This may be deduced from the observation that, by~\eqref{eqn:rho_product}, the function $\rho_1$ is in this case given by~${\rho_1(x)=1/a_1x}$. The value of $_1\psi_1^\ast(x,q/b_1)$ may then be determined from the $q$-binomial theorem, leading to the conclusion that $A(x)$ is identically equal to $(b_1/a_1)_\infty$, from which~\eqref{eqn:normalized_1psi1} follows. The same proof of~\eqref{eqn:1psi1}, expressed in slightly different terms, may be found in~\cite[Section~1.1]{Ito2014}. No such formula as~\eqref{eqn:1psi1} is known for the sum of the general $_2\psi_2$ series. In general, the problem of determining the $A$ and $\rho_j$ of Theorem~\ref{thm:general_rpsir_factorization} explicitly is a highly non-trivial problem to which only a few isolated solutions are known from classical summation formulae such as~\eqref{eqn:1psi1} and~\eqref{eqn:6psi6}. However, as alluded to above, in the case of the $_2\psi_2$ series the function $A$ is expressible in terms of the function $\rho_1$. Moreover, an explicit formula can be given for $\rho_1$ involving an elliptic integral which contains a ratio of $_2\psi_2$ functions in its upper limit. These are the main results of this paper, stated in the following theorem. A special case, pertaining to the Appell--Lerch function, has been given previously by the author in~\cite[Section~3.8]{Bradley-Thrush2023}.

\begin{Theorem} \label{thm:2psi2_factorization}
In the factorization
\begin{equation} \label{eqn:weak_2psi2ast_factorization}
_2^{}\psi_2^\ast(x,y) = A(x) \theta\left({y \over \rho(x)}\right) \theta \bigl( a_1a_2xy\rho(x) \bigr)
\end{equation}
given by the special case $r=2$ of Theorem~{\rm\ref{thm:general_rpsir_factorization}} $($with $\rho$ written here for $\rho_1)$, the functions $A$ and~$\rho$ are related by the formula
\begin{align}
A(x)^2 = {}&\frac{a_1a_2x \rho(x) \bigl(x,{b_1b_2 \over a_1a_2x},{b_1 \over a_1},{b_1 \over a_2},{b_2 \over a_1},{b_2 \over a_2}\bigr)_{ \infty} }{ ((a_1+a_2)x-b_1-b_2) \theta\bigl({\rho(x/q) \over \rho(x)}\bigr)}\nonumber\\
&\times \frac{ \theta\bigl( {\rho(x/q) \over \rho(qx)} \bigr) \theta( a_1a_2x\rho(x/q)\rho(qx)) }{\theta\bigl({\rho(x) \over \rho(qx)}\bigr) \theta (a_1a_2x\rho(x)\rho(x/q) ) \theta (a_1a_2x\rho(x)\rho(qx) )}.\label{eqn:2psi2_A_rho_relation}
\end{align}
Moreover, the function $\rho$ is given by the formula\footnote{There is an apparent ambiguity in this formula arising from the choice of sign made for $\sqrt{a_1a_2x}$. That no such ambiguity arises may be seen from considering the effect of the substitution $u \mapsto 1/u$ on the integral. The function $\rho$ is, however, multivalued since the elliptic integral itself has this property. (See comments regarding this integral toward the end of Section~\ref{sect:2psi2_series}.) The question of which sign is to be taken for $A$ in~\eqref{eqn:2psi2_A_rho_relation} seems less straightforward, but it may be chosen to agree at $x=1$ with that of the value known from~\eqref{eqn:rpsir_factorization_at_x=1}.}
\begin{align}
\rho(x) = {}&{1 \over \sqrt{a_1a_2x}}\exp \bigg( {1 \over \theta\bigl(\sqrt{q}\bigr)\theta(-\sqrt{q})} \nonumber \\
&\times\int_0^{-{_2^{}\psi_2^\ast(x,1/\sqrt{a_1a_2x}) \over _2^{}\psi_2^\ast(x,-1/\sqrt{a_1a_2x})}} {\dd u \over \sqrt{u\Bigl( 1-{\theta(\sqrt{q})^2 \over \theta(-\sqrt{q})^2}u\Bigr)\Bigl( 1-{\theta(-\sqrt{q})^2 \over \theta(\sqrt{q})^2}u\Bigr)}} \Bigg) \label{eqn:rho_integral_formula}
\end{align}
and satisfies the relation
\begin{gather}
{\theta\bigl({\rho(qx) \over \rho(q^3x)}\bigr) \theta\bigl(a_1a_2x\rho(qx)\rho\bigl(q^3x\bigr)\bigr) \over \theta\bigl({\rho(q^2x) \over \rho(q^3x)}\bigr) \theta\bigl(a_1a_2x\rho\bigl(q^2x\bigr)\rho\bigl(q^3x\bigr)\bigr)} \nonumber\\
\qquad= {\bigl(b_1+b_2-(a_1+a_2)qx\bigr) \bigl(b_1+b_2-(a_1+a_2)q^2x\bigr) \theta\bigl({\rho(qx) \over \rho(x)}\bigr) \theta\bigl(a_1a_2x\rho(x)\rho(qx)\bigr)\over (1-qx)\bigl(b_1b_2-a_1a_2q^2x\bigr) \theta\bigl({\rho(q^2x) \over \rho(x)}\bigr) \theta\bigl(a_1a_2x\rho(x)\rho\bigl(q^2x\bigr)\bigr)}. \label{eqn:rho_functional_equation}
\end{gather}
It is analytic everywhere except at isolated branch points which arise as the solutions of the four equations $_2^{}\psi_2^\ast(x,\pm 1/\sqrt{a_1a_2x})=0$ and $_2^{}\psi_2^\ast\bigl(x,\pm \sqrt{q/a_1a_2x}\bigr)=0$.
\end{Theorem}

One of the most interesting aspects of the factorization which is the subject of Theorem~\ref{thm:general_rpsir_factorization} is that it provides one explanation for the importance of very-well-poised-balanced (VWP-balanced) series. The factorizations of these series have a special property: Their zeros occur in pairs in such a way that every factor of $\theta(y/\rho_j(x))$ is accompanied in the factorization by a corresponding factor of $\theta(q\rho_j(x)y)$. Moreover, either three or four of the factors are known explicitly depending on whether $r$ is odd or even. These observations are stated precisely in Theorem~\ref{thm:general_rWr_factorization} below.

The notation
\begin{equation} \label{eqn:rWr_def}
_rW_r(y) := \sum_{n=-\infty}^\infty {(a_1 y, a_2 y, \dots, a_{r-2} y)_n \over (qy/a_1, qy/a_2, \dots, qy/a_{r-2})_n} \bigl(1-y^2q^{2n}\bigr) \bigg({q^{r-4 \over 2} \over a_1a_2 \cdots a_{r-2}}\bigg)^{ n}
\end{equation}
is used for VWP-balanced series throughout this paper; the series converges for all non-zero values of $y$ provided that $|a_1\cdots a_{r-2}|>|q|^{(r-4)/2}$. In some ways, it might be more logical to write $_{r-2}W_{r-2}$ for the series in~\eqref{eqn:rWr_def} so that the subscripts would match the number of $q$-Pochhammer symbols on the numerator and denominator of the summand. The notation chosen here is modelled loosely on that of~\cite[Section~2.1]{Gasper2004}, where, effectively, the factor of $1-y^2q^{2n}$ is introduced in the form $\bigl(1-y^2\bigr)(qy,-qy)_n/(y,-y)_n$.

We may write
\begin{equation} \label{eqn:rWrast_def}
_r^{}W_r(y) = {_r^{}W_r^\ast(y) \over \Bigl({q^{r-4 \over 2} \over a_1a_2\cdots a_{r-2}}, {q \over a_1 y}, {qy \over a_1}, {q \over a_2 y}, {qy \over a_2}, \dots, {q \over a_{r-2} y}, {qy \over a_{r-2}}\Bigr)_{ \infty}},
\end{equation}
where the numerator, $_r^{}W_r^\ast(y)$, is a function of $y$ analytic throughout $\C \backslash \{0\}$. The following result, a refinement of Theorem~\ref{thm:general_rpsir_factorization} for the functions $_r^{}W_r^\ast$, is obtained in Section~\ref{sect:theta}. Its derivation is essentially a matter of combining~\cite[Lemma~4.5]{Ito2008} with~\cite[Proposition~3.4]{Rosengren2006}.

\begin{Theorem} \label{thm:general_rWr_factorization}
The functions $_3^{}W_3^\ast$ and $_4^{}W_4^\ast$ are identically zero. If $r \geq 5$ is odd, then $_r^{}W_r^\ast$ takes the form
\begin{align}
_r^{}W_r^\ast(y) ={}& A \theta(y) \theta(-y) \theta(-y\sqrt{q}) \theta\left( {y \over \rho_1} \right) \theta(q\rho_1y) \theta\left( {y \over \rho_2} \right)\nonumber\\
&\times \theta(q\rho_2y) \cdots \theta\bigg( {y \over \rho_{r-5 \over 2}} \bigg) \theta(q\rho_{r-5 \over 2}y). \label{eqn:rWr_factorization_odd_case}
\end{align}
If $r \geq 6$ is even, then $_r^{}W_r^\ast$ takes the form
\begin{equation} \label{eqn:rWr_factorization_even_case}
_r^{}W_r^\ast(y) = A \theta\bigl(y^2\bigr) \theta\left( {y \over \rho_1} \right) \theta( q\rho_1y) \theta\left( {y \over \rho_2} \right) \theta(q\rho_2y) \cdots \theta\bigg( {y \over \rho_{r-6 \over 2}} \bigg) \theta( q\rho_{r-6 \over 2}y).
\end{equation}
In both cases, the value of $A$ as well as the coefficients $\rho_j$ are independent of the variable $y$ {\rm(}although they will depend in general on the parameters $a_j)$.
\end{Theorem}

\begin{Remark}\quad
\begin{itemize}\itemsep=0pt
\item[(i)] The first part of Theorem~\ref{thm:general_rWr_factorization}---the part which asserts that $_3^{}W_3^\ast= {}_4^{}W_4^\ast=0$---is actually a special case of Bailey's $_6\psi_6$ summation~\eqref{eqn:6psi6}, which may be written in the form
\[
_6^{}W_6^\ast(y) = \left({q \over a_1a_2}, {q \over a_1a_3}, {q \over a_1a_4}, {q \over a_2a_3}, {q \over a_2a_4}, {q \over a_3a_4}\right)_{ \infty} \theta\bigl(y^2\bigr).
\]
Since the $_rW_r$ series defined in~\eqref{eqn:rWr_def} reduces when $a_{r-2}=\sqrt{q}$ to the series $_{r-1}W_{r-1}$, the function $_{r-1}^{}W_{r-1}^\ast$ may in general be recovered from the function $_r^{}W_r^\ast$ using the relation
\[ _r^{}W_r^\ast(y) \Big|_{a_{r-2}=\sqrt{q}} = {\theta(y\sqrt{q}) \over (q)_\infty} {}_{r-1}^{}W_{r-1}^\ast(y). \]
Hence, using the elementary relation
\begin{equation} \label{eqn:theta_squared-argument_identity}
\theta\bigl(y^2\bigr) = {1 \over (q)_\infty^3} \theta(y) \theta(-y) \theta(y\sqrt{q}) \theta(-y\sqrt{q}),
\end{equation}
it follows that
\[ {}_5^{}W_5^\ast(y) = {(q)_\infty {}_6^{}W_6^\ast(y) \Big|_{a_4=\sqrt{q}} \over \theta(y\sqrt{q})} = {\bigl({\sqrt{q} \over a_1}, {\sqrt{q} \over a_2}, {\sqrt{q} \over a_3}, {q \over a_1a_2}, {q \over a_1a_3}, {q \over a_2a_3}\bigr)_\infty \over (q)_\infty^2} \theta(y) \theta(-y) \theta(-y\sqrt{q}), \]
and consequently
\[ _4^{}W_4^\ast(y) = {(q)_\infty {}_5^{}W_5^\ast(y) \Big|_{a_3=\sqrt{q}} \over \theta(y\sqrt{q})} = 0, \qquad \text{and} \qquad _3^{}W_3^\ast(y) = {(q)_\infty {}_4^{}W_4^\ast(y) \Big|_{a_2=\sqrt{q}} \over \theta(y\sqrt{q})} = 0. \]
In the proof of Theorem~\ref{thm:general_rWr_factorization} given in Section~\ref{sect:theta}, the fact that $_3^{}W_3^\ast$ and $_4^{}W_4^\ast$ are identically zero is established by a different argument which does not make use of~\eqref{eqn:6psi6}.
\item[(ii)] The special case $r=8$ of Theorem~\ref{thm:general_rWr_factorization} asserts that
\begin{equation} \label{eqn:weak_8psi8ast_factorization}
_8^{}W_8^\ast(y) = A \theta\bigl(y^2\bigr) \theta\left( {y \over \rho} \right) \theta( q\rho y)
\end{equation}
for some $A$ and $\rho$ which are independent of $y$. A relation exists between these two unknowns, each of which may be regarded as a function of the parameter $a_1$. This relation, analogous to the result of Theorem~\ref{thm:2psi2_factorization}, will be developed in a later paper. From Gosper's bilateral Jackson formula (as it is termed by Gasper and Rahman~\cite[Exercise~5.12]{Gasper2004}), it is apparent that the relation between the $A$ and $\rho$ of~\eqref{eqn:weak_8psi8ast_factorization} is simplified significantly when $a_1a_2a_3a_4a_5a_6=q$. Supposing this to be the case, $A$ is given in terms of $\rho$ by the formula
\[ A = {\bigl({q \over a_1a_3},{q \over a_1a_4},{q \over a_1a_5},{q \over a_1a_6},{q \over a_3a_4},{q \over a_3a_5},{q \over a_3a_6},{q \over a_4a_5},{q \over a_4a_6},{q \over a_5a_6}\bigr)_{ \infty} \theta(a_2\xi) \theta\bigl({a_2 \over \xi}\bigr) \over (q,a_1a_2,a_2a_3,a_2a_4,a_2a_5,a_2a_6)_{ \infty} \theta(q\rho\xi)\theta\bigl({\xi \over \rho}\bigr)} ,\]
where
\begin{align*}
\xi={}& {1 \over a_1}\exp \Bigg( {\theta(a_1a_2) \theta\bigl({a_1 \over a_2}\bigr) \over \theta\bigl(a_1^2\bigr) (q)_\infty^3} \\
&\times \int_0^{{a_2\theta(a_1a_3)\theta(a_1a_4)\theta(a_1a_5)\theta(a_1a_6) \over a_1\theta(a_2a_3)\theta(a_2a_4)\theta(a_2a_5)\theta(a_2a_6)}}\frac{1}{\sqrt{\Bigl(1-{a_1\theta(a_2)^2 \over a_2\theta(a_1)^2}u\Bigr) \Bigl(1-{a_1\theta(-a_2)^2 \over a_2\theta(-a_1)^2}u\Bigr)}} \\
&\times {\dd u \over \Bigl(1-{\theta(a_2\sqrt{q})^2 \over \theta(a_1\sqrt{q})^2}u\Bigr) \Bigl(1-{\theta(-a_2\sqrt{q})^2 \over \theta(-a_1\sqrt{q})^2}u\Bigr)} \Bigg).
\end{align*}
\end{itemize}
\end{Remark}

\section{Applications}

Aside from an obvious desire to generalize the classical summation formulae, there is another reason why investigating the general problem of factorization of basic hypergeometric series is particularly worthwhile: It provides some insight into, and in many cases simpler proofs of, a~number of well-known identities. The following example is taken from the author's thesis~\cite[Section~3.6]{Bradley-Thrush2023}. By Theorem~\ref{thm:general_rpsir_factorization}, we know that the numerator, $_2^{}\psi_2^\ast$, of the general $_2\psi_2$ series must take the form~\eqref{eqn:weak_2psi2ast_factorization} for some unknown functions $A$ and $\rho$. The relation between these two functions provided by Theorem~\ref{thm:2psi2_factorization} need not concern us here. Since the function $\theta$ satisfies the elementary identity $\theta(x) = \theta(q/x)$, it follows at once from~\eqref{eqn:weak_2psi2ast_factorization} that
${}_2^{}\psi_2^\ast(x,y) = {}_2^{}\psi_2^\ast(x,q/a_1a_2xy)$.
This is the second of a pair of $_2\psi_2$ transformations found by Bailey~\cite[equations~(2.3) and~(2.4)]{Bailey1950}. His other $_2\psi_2$ transformation does not follow from~\eqref{eqn:weak_2psi2ast_factorization} alone; rather, it follows from a symmetry of the functions $A$ and $\rho$ which is not at all obvious: regarded as (multi-valued) functions of the variables $a_1$, $a_2$, $b_1$, $b_2$, $x$, the functions $A$ and $\rho$ in~\eqref{eqn:weak_2psi2ast_factorization} are both invariant under the substitution $(a_2,b_2,x) \mapsto (a_1a_2x/b_2,a_1x,b_2/a_1)$.

In the previous section, it was observed that Ramanujan's $_1\psi_1$ summation~\eqref{eqn:1psi1} follows immediately from Theorem~\ref{thm:general_rpsir_factorization} and the $q$-binomial theorem. In much the same way, Bailey's~$_6\psi_6$ summation~\eqref{eqn:6psi6} may be obtained from Theorem~\ref{thm:general_rWr_factorization} and the $_6\phi_5$ summation (for which see~\cite[equation~(2.7.1)]{Gasper2004}). Consider that the special case $r=6$ of Theorem~\ref{thm:general_rWr_factorization} asserts that~${_6W_6^\ast(y) = A \theta\bigl(y^2\bigr)}$ for some $A$ which is independent of~$y$. The value of $A$ may then be determined by setting $y=a_1$, since the value of $_6W_6^\ast(a_1)$ is readily obtained from the $_6\phi_5$ summation formula.

A further application of the ideas of this paper is to the derivation of Slater's general transformation formulae for bilateral basic hypergeometric series~\cite{Slater1952}, which may be written in the form
\begin{equation} \label{eqn:Slater_general_linear_bilateral_relation}
_r^{}\psi_r^\ast(x,y) = {1 \over \theta(a_1z_1\cdots a_r z_r x)} \sum_{j=1}^r {}_r^{}\psi_r^\ast(x,z_j) \theta \left({a_1z_1\cdots a_r z_r xy \over z_j} \right) \prod_{\substack{k=1 \\ k \neq j}}^r {\theta\bigl({y \over z_k}\bigr) \over \theta\bigl({z_j \over z_k}\bigr)}
\end{equation}
and, for even $r \geq 6$,
\begin{equation} \label{Slater_VWP_linear_bilateral_relation_even}
{_r^{}W_r^\ast(y) \over \theta\bigl(y^2\bigr)} = \sum_{j=0}^{r-6 \over 2} {_r^{}W_r^\ast(z_j) \over \theta(z_j^2)} \prod_{\substack{k=0 \\ k \neq j}}^{r-6 \over 2} {\theta\bigl({y \over z_k}\bigr) \theta(qz_ky) \over \theta\bigl({z_j \over z_k}\bigr) \theta(qz_jz_k)},
\end{equation}
while for odd $r \geq 5$,
\begin{equation} \label{Slater_VWP_linear_bilateral_relation_odd}
{\theta(y\sqrt{q}) {}_r^{}W_r^\ast(y) \over \theta\bigl(y^2\bigr)} = \sum_{j=0}^{r-5 \over 2} {\theta(z_j\sqrt{q}) {}_r^{}W_r^\ast(z_j) \over \theta\bigl(z_j^2\bigr)} \prod_{\substack{k=0 \\ k \neq j}}^{r-5 \over 2} {\theta\bigl({y \over z_k}\bigr) \theta(qz_ky) \over \theta\bigl({z_j \over z_k}\bigr) \theta(qz_jz_k)}.
\end{equation}
As explained in Remarks~\ref{rmk:Slater_general_linear_bilateral_relation} and~\ref{rmk:Slater_VWP_linear_bilateral_relations}, these follow directly from lemmas in the following section, which imply that, as functions of $y$, both $_r^{}\psi_r^\ast(x,y)$ and $_r^{}W_r^\ast(y)$ belong to spaces of finite dimension. Essentially the same derivation of these identities is given by Ito and Sanada~\cite[Section~3.1]{Ito2008}.

\section{Theta functions} \label{sect:theta}

This section contains several lemmas which are needed in Section~\ref{sect:2psi2_series}. With the possible exception of Lemmas~\ref{lem:homogeneous_first-order_q-difference_solution} and~\ref{lem:Wronskian}, all of the results are classical. Lemma~\ref{lem:theta_factorization}, for example, may be found in the book of Briot and Bouquet~\cite[p.~239]{Briot1875}, albeit in a rather different notation. A more modern statement of it appears in~\cite[Lemma~3.2]{Rosengren2006}; functions belonging to the space $\Theta_n(c)$ defined below are there called $A_{n-1}$ theta functions. A special case of Lemma~\ref{lem:theta_determinant_lemma} was given by Appell~\cite[p.~140]{Appell1884}.

\begin{Definition}
Let $n$ be any integer and let $c$ be any non-zero complex number. Let $\Theta_n(c)$ denote the space of analytic functions $f\colon \C \backslash \{0\} \rightarrow \C$ which satisfy identically the relation
\begin{equation} \label{eqn:nth_order_theta_functional_equation}
f(qx) = {(-1)^n f(x) \over cx^n}.
\end{equation}
\end{Definition}

\begin{Lemma}[classical] \label{lem:Theta_dimension}
Let $n$ be an integer and let $c$ be any non-zero complex number. If $n$ is negative, then $\Theta_n(c) = \{0\}$. If $n$ is positive, then $\dim \Theta_n(c) = n$. If $n=0$, then there are two possibilities:
\[ \Theta_0(c) = \begin{cases} \big\{f \mid f(x) = A x^{-k} \text{ for some } A \in \C\big\} &\text{if $c = q^k$ for some integer $k$,} \\
0 &\text{otherwise.}
\end{cases} \]
\end{Lemma}

\begin{Remark}
In particular, Lemma~\ref{lem:Theta_dimension} asserts that $\Theta_0(1)$ consists only of constant functions, i.e., that any solution of $f(qx)=f(x)$ which is analytic on $\C\backslash\{0\}$ must be constant. This special case is well known and widely used in proofs of $q$-series identities.
\end{Remark}

\begin{Lemma} \label{lem:theta_root}
Let $n$ be any integer, let $c$ be any non-zero complex number, and let $f \in \Theta_n(c)$. If $\rho$ is a non-zero complex number such that $f(\rho)=0$, then $f(x) = \theta(x/\rho)g(x)$ for some $g \in \Theta_{n-1}(c\rho)$.
\end{Lemma}

\begin{proof}
From the relation~\eqref{eqn:nth_order_theta_functional_equation}, it follows that $f(\rho)=0$ implies $f\bigl(\rho q^k\bigr)=0$ for every integer~$k$. Consequently the function $g$ defined by $g(x) = f(x)/\theta(x/\rho)$ is analytic throughout $\C \backslash \{0\}$. Since
\[ g(qx) = {(-1)^n f(x) / cx^n \over (-\rho/x) \theta(x/\rho)} = {(-1)^{n-1} g(x) \over c\rho x^{n-1}}, \]
it follows that $g \in \Theta_{n-1}(c\rho)$.
\end{proof}

\begin{Lemma}[classical] \label{lem:theta_factorization}
Let $n$ be a positive integer and let $c$ be any non-zero complex number. Then $f \in \Theta_n(c)$ if and only if
$f(x) = A \sp \theta(\alpha_1x) \theta(\alpha_2x) \cdots \theta(\alpha_nx)$
for some $A \in \C$ and some~${\alpha_1, \alpha_2,\dots,\alpha_n \in \C\backslash\{0\}}$ such that $\alpha_1\alpha_2 \cdots \alpha_n = c$.
\end{Lemma}

\begin{proof}[Proof of Theorem~\ref{thm:general_rpsir_factorization}]
By Lemma~\ref{lem:theta_factorization}, it suffices to verify that, regarded as a function of the variable $y$, $_r^{}\psi_r^\ast(x,y)$ belongs to the space $\Theta_r(a_1a_2\cdots a_r x)$. Since this function is analytic on~${\C \backslash\{0\}}$, it suffices to check that it satisfies the relation
\[ _r^{}\psi_r^\ast(x,qy) = {(-1)^r {}_r^{}\psi_r^\ast(x,y) \over a_1a_2\cdots a_r x y^r}. \]
This is easily verified by changing $n \mapsto n+1$ in the series in~\eqref{eqn:rpsir_def}.
\end{proof}

\begin{Remark} \label{rmk:Theta2_relation}
If $f \in \Theta_2(c)$, then $f$ satisfies identically the relation $f(x) = f(q/cx)$. The reason for this is that, by Lemma~\ref{lem:theta_factorization}, any such function $f$ may be written in the form $f(x) = A\theta(\alpha x) \theta(cx/\alpha)$.
\end{Remark}

\begin{Remark} \label{rmk:Slater_general_linear_bilateral_relation}
Let $n$ be a positive integer, and $c$ any non-zero complex number. Given any $z_1,z_2,\dots,z_n \in \C \backslash\{0\}$, for $1 \leq j \leq n$ let
\[ \vartheta_j(x) = {\theta\bigl({cz_1 \cdots z_n x \over z_j}\bigr) \over \theta(cz_1\cdots z_n)} \prod_{\substack{k=1 \\ k \neq j}}^n {\theta\bigl({x \over z_k}\bigr) \over \theta\bigl({z_j \over z_k}\bigr)}. \]
Provided that the numbers $z_j$ are such that the denominators of these functions are not zero, the relation $\vartheta_j(z_k)=\delta_{j,k}$ holds for all pairs of integers $j$, $k$ in the interval $[1,n]$. These $n$ functions are therefore linearly independent of one another and so form a basis for $\Theta_n(c)$. Consequently, any given $f \in \Theta_n(c)$ may be expanded in the form
\begin{equation} \label{eqn:general_Theta_expansion}
f(x) = \sum_{j=1}^n f(z_j) \vartheta_j(x).
\end{equation}
This expansion is equivalent to a special case of a formula given by Appell~\cite[equations~(8) and~($8'$)]{Appell1884}. As is commented upon in~\cite[p.~88]{Bradley-Thrush2023} and in~\cite[Section~5.1]{Ito2008}, Slater's identity~\eqref{eqn:Slater_general_linear_bilateral_relation} is an instance of~\eqref{eqn:general_Theta_expansion}.
\end{Remark}

The results stated in Lemmas~\ref{lem:omega_dim}, \ref{lem:Omega_characterization}, and~\ref{lem:Omega_characterization_with_sign_changed} relate to a particular subspace of $\Theta_{2n}(c^n)$, introduced in Definition~\ref{def:Omega} below. These results may be found in~\cite[Lemma~4.3]{Ito2008} and~\cite[Lemma~3.2]{Rosengren2006}.\footnote{The original version of this paper, written before the author became aware of~\cite{Ito2008} and \cite{Rosengren2006}, included proofs of these results from scratch. These have been removed from the present version as the proofs in those other papers are shorter and more direct.} In the terminology used there, the functions within the subspace are said to be~$D_{n+1}$ theta functions, and functions satisfying the conditions of Lemma~\ref{lem:Omega_characterization_with_sign_changed} are said to be~$C_{n-1}$ theta functions. In~\cite{Ito2008} and \cite{Rosengren2006}, essentially the characterization given in Lemma~\ref{lem:Omega_characterization} is taken as the definition of the subspace in question.

\begin{Definition} \label{def:Omega}
Let $n$ be a positive integer and let $c$ be any non-zero complex number. Let~$\Omega_{2n}(c)$ denote the subset of $\Theta_{2n}(c^n)$ consisting of functions $f$ which take the form
\[
f(x) = \prod_{j=1}^n f_j(x)
\]
with each $f_j \in \Theta_2(c)$. It is convenient also to define $\Omega_0(c)$ to be the space of constant functions (i.e., $\Theta_0(1)$), and further to define $\Omega_{2n}(c) = \{0\}$ for every $n<0$.
\end{Definition}

\begin{Lemma} \label{lem:omega_dim}
Let $n$ be a non-negative integer and $c$ any non-zero complex number. Then the set $\Omega_{2n}(c)$ is in fact a subspace of $\Theta_{2n}(c^n)$, and its dimension equals $n+1$.
\end{Lemma}

\begin{Remark}\quad
\begin{itemize}\itemsep=0pt
\item[(i)] Although $\Omega_{2n}(c) = \Theta_{2n}(c^n)$ for every $n \leq 1$, the inclusion $\Omega_{2n}(c) \subset \Theta_{2n}(c^n)$ is strict for every $n \geq 2$ by Lemmas~\ref{lem:Theta_dimension} and~\ref{lem:omega_dim}.
\item[(ii)] For $n \geq 1$, it is apparent that $f \in \Omega_{2n}(c)$ if and only if there is some $c' \neq 0$ such that $f(x) = g(x) g(q/cx)$ for some $g \in \Theta_n(c')$. Another characterization of the space $\Omega_{2n}(c)$ is given in the following lemma.
\end{itemize}
\end{Remark}

\begin{Lemma} \label{lem:Omega_characterization}
Let $n$ be any integer, let $c$ be a non-zero complex number, and let $f \in \Theta_{2n}(c^n)$. Then $f \in \Omega_{2n}(c)$ if and only if the relation
\begin{equation} \label{eqn:2nth_order_theta_reciprocal_relation}
f(x) = f(q/cx)
\end{equation}
holds identically.
\end{Lemma}

\begin{Lemma} \label{lem:Omega_characterization_with_sign_changed}
Let $n$ be any integer, let $c$ be a non-zero complex number, and let $f \in \Theta_{2n}(c^n)$. Then $f$ satisfies identically the relation
\begin{equation} \label{eqn:2nth_order_theta_reciprocal_relation_with_minus}
f(x) = -f(q/cx)
\end{equation}
if and only if $f(x) = x\theta\big(cx^2\big) g(x)$ for some $g \in \Omega_{2n-4}(c)$.
\end{Lemma}

\begin{Remark}
The special case $n=1$ of Lemma~\ref{lem:Omega_characterization_with_sign_changed} asserts that any function $f \in \Theta_2(c)$ satisfying the relation~\eqref{eqn:2nth_order_theta_reciprocal_relation_with_minus} must be identically zero. This is apparent from the fact that~\eqref{eqn:2nth_order_theta_reciprocal_relation} holds for every $f \in \Theta_2(c)$.
\end{Remark}

\begin{proof}[Proof of Theorem~\ref{thm:general_rWr_factorization}]
The function $_r^{} W_r^\ast$ defined by~\eqref{eqn:rWrast_def} satisfies the two relations
\[
_r^{}W_r^\ast(qy) = {(-1)^r {}_r^{}W_r^\ast(y) \over q^{r-4 \over 2} y^{r-2}}, \qquad{} _r^{}W_r^\ast(1/y) = -{ {}_r^{}W_r^\ast(y) \over y^2}.
\]
These are easily obtained by changing $n \mapsto n+1$ and $n \mapsto -n$ in the series in~\eqref{eqn:rWr_def}. For the remainder of this proof, let us put $f(y) = {}_r^{}W_r^\ast(y)/y$. Then the first equation shows that $f$ belongs to the space $\Theta_{r-2}\bigl(q^{(r-2)/2}\bigr)$, while the second, which we may rewrite as $f(y)=-f(1/y)$, allows us to infer that $f$ belongs in fact to a subspace thereof which we may determine.

Consider first the case in which $r$ is even. Then, by Lemma~\ref{lem:Omega_characterization_with_sign_changed}, it must be that $f(y)=y\theta\bigl(qy^2\bigr)g(y)$ for some $g \in \Omega_{r-6}(q)$. If $r=4$, then this space contains only the function which is identically zero, so it follows that $_4^{}W_4^\ast(y)=0$ for all $y$. If, on the other hand, $r \geq 6$, then, per Definition~\ref{def:Omega}, we may write \smash{$g(y) = \prod_{j=1}^{(r-6)/2} g_j(y)$}, where each of the functions $g_j$ belongs to the space $\Theta_2(q)$ and therefore, by Lemma~\ref{lem:theta_factorization}, admits a factorization of the form~${g_j(y) = A_j \theta(y/\rho_j)\theta(q\rho_j y)}$. Hence
\[ {_r^{}W_r^\ast(y) \over y} = y\theta\bigl(qy^2\bigr) \prod_{j=1}^{r-6 \over 2} A_j \theta\left({y \over \rho_j}\right)\theta(q\rho_j y). \]
This is exactly the form of the factorization specified in~\eqref{eqn:rWr_factorization_even_case}, the constant $A$ there being equal to \smash{$-\prod_{j=1}^{(r-6)/2} A_j$} in the notation used here.

Consider next the case in which $r$ is odd, and put $F(y)\! =\! \theta(y\sqrt{q}) f(y)$. Then $F \!\in \!\Theta_{r-1}\bigl(q^{(r-1)/2}\bigr)$ and, since $\theta(y\sqrt{q})$ is invariant under $y \mapsto 1/y$, this function $F$ satisfies $F(y)=-F(1/y)$. Another application of Lemma~\ref{lem:Omega_characterization_with_sign_changed} allows us to conclude that $F(y) = y\theta\bigl(qy^2\bigr)g(y)$ for some $g \in \Omega_{r-5}(q)$. Hence $_3^{}W_3^\ast$ is identically zero and, for odd $r \geq 5$,
\[ {\theta(y\sqrt{q}) {}_r^{}W_r^\ast(y) \over y} = y\theta\bigl(qy^2\bigr) \prod_{j=1}^{r-5 \over 2} A_j \theta\left({y \over \rho_j}\right)\theta(q\rho_j y). \]
In view of the identity~\eqref{eqn:theta_squared-argument_identity}, this is exactly the form of the factorization specified in~\eqref{eqn:rWr_factorization_odd_case}.
\end{proof}

\begin{Remark} \label{rmk:Slater_VWP_linear_bilateral_relations}
Let $n$ be a positive integer and let $c$ be a non-zero complex number. Given any~${n+1}$ non-zero complex numbers $z_0, z_1,\dots, z_n$, for $0 \leq j \leq n$ let
\[ \omega_j(x) = \prod_{\substack{k=0 \\ k \neq j}}^n {\theta(x/z_k) \theta(cz_kx) \over \theta(z_j/z_k) \theta(cz_jz_k)}. \]
Then, provided that the $z_j$ are such that the denominators are non-zero, $\omega_j \in \Omega_{2n}(c)$ for each $j$. Moreover, these functions $\omega_j$ are linearly independent since $\omega_j(z_k) = \delta_{j,k}$ for all integers $j$, $k$ in the interval $[0,n]$. By Lemma~\ref{lem:omega_dim}, these functions form a basis for $\Omega_{2n}(c)$. Hence any $f \in \Omega_{2n}(c)$ may be expressed in the form
\begin{equation} \label{eqn:Omega_basis_expansion}
f(x) = \sum_{j=0}^n f(z_j) \omega_j(x).
\end{equation}
In the proof of Theorem~\ref{thm:general_rWr_factorization}, it has been shown that, when $r$ is even, the function $_r^{}W_r^\ast(y)/\theta\bigl(y^2\bigr)$ belongs to the space $\Omega_{r-6}(q)$, and that, when $r$ is odd, the function $\theta(y\sqrt{q}){}_r^{}W_r^\ast(y)/\theta\bigl(y^2\bigr)$ belongs to the space $\Omega_{r-5}(q)$. In either case, this allows the function $_r^{}W_r^\ast$ to be expanded with respect to the basis functions $\omega_j$ just described; doing so yields Slater's identities~\eqref{Slater_VWP_linear_bilateral_relation_even} and~\eqref{Slater_VWP_linear_bilateral_relation_odd}.

The expansion~\eqref{eqn:Omega_basis_expansion} may be used to obtain a formula valid for all $f \in \Theta_n(c')$, where $c'$ may denote any non-zero complex number. Given such a function $f$, it is straightforward to see that the function $f(x)f(q/cx)$ belongs to the space $\Omega_{2n}(c)$, and so
\[ f(x)f\left({q \over cx}\right) = \sum_{j=0}^n f(z_j)f\left({q \over cz_j}\right) \omega_j(x), \]
with the basis functions $\omega_j$ as above. In particular, this implies the identity
\[ {}_r^{}\psi_r^\ast(x,y) {}_r^{}\psi_r^\ast\left(x,{q \over cy}\right) = \sum_{j=0}^r {}_r^{}\psi_r^\ast(x,z_j) {}_r^{}\psi_r^\ast\left(x,{q \over cz_j}\right) \prod_{\substack{k=0 \\ k \neq j}}^r {\theta(y/z_k) \theta(cz_ky) \over \theta(z_j/z_k) \theta(cz_jz_k)}, \]
valid for every positive integer $r$.
\end{Remark}

\begin{Definition}
For each positive integer $n$, let $\Delta_n\colon (\C \backslash \{0\})^n \rightarrow \C$ denote the product
\[ \Delta_n(x_1,x_2,\dots,x_n) = \prod_{1 \leq j < k \leq n} x_k \theta\left( {x_j \over x_k} \right), \]
with the convention that the function $\Delta_1$ (defined by an empty product) is identically equal to~$1$.
\end{Definition}

The following result can be found in~\cite[Proposition~3.4]{Rosengren2006}. Its proof is straightforward: since~$\Delta_n$ is anti-symmetric, the ratio of the two sides of~\eqref{eqn:q-Vandermonde_determinant} is a symmetric function of~$x_1, x_2, \allowbreak\dots, x_n$; moreover, it follows easily from Lemmas~\ref{lem:theta_root} and~\ref{lem:theta_factorization} that this ratio is independent of $x_1$.

\begin{Lemma} \label{lem:theta_determinant_lemma}
Let $n$ be a positive integer, $c$ any non-zero complex number, and let $f_1,f_2,\dots,f_n\allowbreak \in \Theta_n(c)$. Then there is some constant $A \in \C$ such that the identity
\begin{equation} \label{eqn:q-Vandermonde_determinant}
\det_{1 \leq i,j \leq n}(f_i(x_j))
= A \Delta_n(x_1,x_2,\dots,x_n) \theta(cx_1x_2 \cdots x_n)
\end{equation}
holds for all $x_1,x_2,\dots,x_n \in \C \backslash\{0\}$.
\end{Lemma}

\begin{Lemma} \label{lem:homogeneous_first-order_q-difference_solution}
Let $n$ be any integer and let $c$ be any non-zero complex number. Let $\varphi_1$ and $\varphi_2$ be any two entire functions with the following properties:
\begin{enumerate}\itemsep=0pt
\item[$(i)$] $\varphi_1(0)=\varphi_2(0)=1$;
\item[$(ii)$] 
if $\alpha,\beta \in \C$ satisfy $\varphi_1(\alpha)=\varphi_2(\beta)=0$, then $\alpha \beta \neq q^m$ for any integer $m$.
\end{enumerate}
Then an analytic function $f\colon\C\backslash\{0\}\rightarrow\C$ satisfies identically the relation
\begin{equation} \label{eqn:homogeneous_first-order_q-difference}
f(qx) = {(-1)^n \varphi_2(1/x) \over cx^n \varphi_1(x)} f(x)
\end{equation}
if and only if
\begin{equation} \label{eqn:homogeneous_first-order_q-difference_solution}
f(x) = g(x) \prod_{m=1}^\infty \varphi_1\bigl(xq^{m-1}\bigr) \varphi_2(q^m/x)
\end{equation}
for some $g \in \Theta_n(c)$.
\end{Lemma}

\begin{proof}
First, observe that the product in~\eqref{eqn:homogeneous_first-order_q-difference_solution} defines a function analytic throughout $\C\backslash\{0\}$. To verify this, put
\begin{equation} \label{eqn:Phi_def}
\Phi(x) = \prod_{m=1}^\infty \varphi_1\bigl(xq^{m-1}\bigr) \varphi_2(q^m/x)
\end{equation}
and let $r_1$ and $r_2$ be any two positive numbers such that $r_2 < r_1$. The two functions $(\varphi_1(x)-1)/x$ and $(\varphi_2(x)-1)/x$ are entire. They are therefore bounded on any given disk about the origin. It is therefore possible to choose a positive constant $C$ such that $|\varphi_1(x)-1| \leq C|x|$ for $|x| \leq r_1$ and $|\varphi_2(x)-1| \leq C|x|$ for $|x| \leq |q|/r_2$. Hence, if $x$ is such that $r_2 \leq |x| \leq r_1$, then, for any positive integer $m$,
$\big| \varphi_1\bigl(xq^{m-1}\bigr) - 1 \big| \leq Cr_1|q|^{m-1}$ and $|\varphi_2(q^m/x) - 1| \leq C|q|^m/r_2$.
It follows from this that the series $\sum_{m=1}^\infty \big|\varphi_1\bigl(xq^{m-1}\bigr) - 1\big|$ and $\sum_{m=1}^\infty |\varphi_2(q^m/x) - 1|$ converge uniformly on the annulus $r_2 \leq |x| \leq r_1$. The product in~\eqref{eqn:Phi_def} is therefore uniformly convergent on the same region, so the function $\Phi$ is analytic there. Since $r_1$ and $r_2$ are arbitrary (except for the condition $r_2<r_1$), it follows that $\Phi$ is analytic on $\C\backslash\{0\}$.

Now suppose that~\eqref{eqn:homogeneous_first-order_q-difference_solution} holds for some $g \in \Theta_n(c)$. By the foregoing, the function $f$ is analytic throughout $\C\backslash\{0\}$. Since $\Phi(qx) = \varphi_2(1/x) \Phi(x) / \varphi_1(x)$ and $g(qx) = (-1)^n g(x)/cx^n$, it is plain that $f$ must satisfy the relation~\eqref{eqn:homogeneous_first-order_q-difference}.

Suppose next, conversely, that the relation~\eqref{eqn:homogeneous_first-order_q-difference} holds. It is convenient to introduce here the notation $M(\alpha,\phi)$ for the multiplicity of a given complex number $\alpha$ as a zero of a given function~$\phi$, with the understanding that $M(\alpha,\phi)=0$ if $\phi(\alpha) \neq 0$. Suppose that $\alpha$ is a zero of the function~$\varphi_1$. Since $\varphi_1(0)=1$ and $\varphi_1$ is continuous at the origin, the value of $M(\alpha q^m,\varphi_1)$ must be zero for all sufficiently large integer values of $m$. Let $k$ be the greatest integer such that~$M\bigl(\alpha q^{k},\varphi_1\bigr)>0$. Then plainly, on account of condition~(ii),
\[ M\bigl(\alpha q^j,\Phi\bigr) = \sum_{m=j}^k M(\alpha q^m,\varphi_1) \]
for every integer $j$, with the understanding that this sum is to be interpreted as zero whenever~${j>k}$. Also, for any integer $m$, the relation~\eqref{eqn:homogeneous_first-order_q-difference} implies that
\[ M\bigl(\alpha q^{m+1},f\bigr) + M(\alpha q^m,\varphi_1) = M(\alpha q^m,f) + M(q^{-m}/\alpha,\varphi_2). \]
By condition~(ii), the value of $M(q^{-m}/\alpha,\varphi_2)$ is zero. Consequently, for any integer $j$ such that~${j \leq k}$,
\begin{align*}
 M\bigl(\alpha q^j,f\bigr) &= M\bigl(\alpha q^{k+1},f\bigr) + \sum_{m=j}^{k} \bigl( M(\alpha q^m,f) - M\bigl(\alpha q^{m+1},f\bigr) \bigr) \geq \sum_{m=j}^{k} M(\alpha q^m,\varphi_1)\\
 & = M\bigl(\alpha q^j,\Phi\bigr).
 \end{align*}
Hence the ratio $f/\Phi$ is analytic at $x=\alpha q^j$ for every integer $j$. Now suppose that $\beta$ is a zero of the function $\varphi_2$. As before, it is possible to choose a non-negative integer $\ell$ maximal such that~${M\bigl(\beta q^\ell,\varphi_2\bigr)>0}$. Condition~(ii) implies that, for every integer $j$,
\[ M\bigl(q^j/\beta,\Phi\bigr) = \sum_{m=1-j}^{\ell} M(\beta q^{m},\varphi_2). \]
As before, it should be understood here that the sum equals zero whenever it is empty, i.e., when $j < 1-\ell$. For any integer $m$, the relation~\eqref{eqn:homogeneous_first-order_q-difference} yields
\[ M\bigl(q^{1-m}/\beta,f\bigr) + M(q^{-m}/\beta,\varphi_1) = M(\beta q^m,\varphi_2) + M(q^{-m}/\beta,f), \]
and again $M(q^{-m}/\beta,\varphi_1)=0$ on account of condition~(ii), so for any $j$ such that $j \leq 1-\ell$,
\begin{align*}
M\bigl(q^j/\beta,f\bigr) &= M\bigl(q^{-\ell}/\beta,f\bigr) + \sum_{m=1-j}^{\ell} \bigl( M\bigl(q^{1-m}/\beta,f\bigr) - M(q^{-m}/\beta,f) \bigr) \\
&\geq \sum_{m=1-j}^{\ell} M(\beta q^m,\varphi_2) = M\bigl(q^j/\beta,\Phi\bigr).
\end{align*}
Hence $f/\Phi$ is analytic at $x=q^j/\beta$ for every integer $j$. Hence $f/\Phi$ is in fact analytic throughout the region $\C\backslash\{0\}$. Setting $g=f/\Phi$, this function $g$ satisfies the relation
\[ g(qx) = {f(qx) \over \Phi(qx)} = {(-1)^n \varphi_2(1/x) f(x) / (cx^n \varphi_1(x)) \over \varphi_2(1/x) \Phi(x)/\varphi_1(x)} = {(-1)^n g(x) \over cx^n}, \]
so $g \in \Theta_n(c)$ as required.
\end{proof}

The following elementary lemma has been given in a slightly different but equivalent form by Abu Risha et al.\ in~\cite[Theorem~3.3]{AbuRisha2007}. There the $q$-Wronskian is defined in such a way that its elements are $q$-derivatives, so as to make it a $q$-analogue of the Wronskian for an ordinary differential equation. The $q$-Wronskian as it is defined below is readily seen to be related to the one given there by row and column operations.

\begin{Lemma} \label{lem:Wronskian}
Let $n$ be a positive integer. If the functions $f_1, f_2, \dots, f_n$ are such that
\[ \sum_{m=0}^n c_m(x) f_j(xq^m) = 0 \]
for $j=1,2,\dots,n$, then their $q$-Wronskian, $\mathcal{W}(x) = \det_{1\leq i,j \leq n} \bigl( f_i\bigl(xq^{j-1}\bigr) \bigr)$, satisfies the relation
\[ \mathcal{W}(qx) = {(-1)^n c_0(x) \over c_n(x)} \mathcal{W}(x). \]
\end{Lemma}

\section[The general \_2psi\_2 series]{The general $\boldsymbol{_2\psi_2}$ series} \label{sect:2psi2_series}

\begin{proof}[Proof of Theorem~\ref{thm:2psi2_factorization}]
The identity
\begin{gather}
\sum_{n=-\infty}^\infty {(a_1,a_2)_n \over (b_1,b_2)_n} \bigl(1-b_1q^{n-1}\bigr)\bigl(1-b_2q^{n-1}\bigr) x^n\nonumber\\
\qquad = \sum_{n=-\infty}^\infty {(a_1,a_2)_n \over (b_1,b_2)_n} (1-a_1q^n)(1-a_2q^n) x^{n+1},\label{eqn:elementary_2psi2_shift}
\end{gather}
easily verified by changing $n \mapsto n+1$ on the left-hand side, gives us the well-known three-term $q$-difference equation satisfied by the $_2\psi_2$ function with respect to its variable $x$, namely
\begin{gather*}
 (1-x)\,{}_2\psi_2(x,y) + y\bigg( (a_1+a_2)x - {b_1+b_2 \over q} \bigg) {}_2\psi_2(qx,y) \\
 \qquad{} + y^2 \bigg( {b_1b_2 \over q^2} - a_1a_2x \bigg) {}_2\psi_2\bigl(q^2x,y\bigr)
 =0.
 \end{gather*}
Equivalently, in terms of the function's numerator, ${}_2^{}\psi_2^\ast$, this may be written as
\begin{gather}
\left(1-{b_1b_2 \over a_1a_2qx}\right){}_2^{}\psi_2^\ast(x,y) + y\left( (a_1+a_2)x - {b_1+b_2 \over q} \right) {}_2^{}\psi_2^\ast(qx,y)\nonumber \\
\qquad{}-a_1a_2xy^2 (1-qx) {}_2^{}\psi_2^\ast\bigl(q^2x,y\bigr) =0.\label{eqn:2psi2ast_functional_eqn_wrt_x}
\end{gather}
In this equation, taking $y$ equal to $\rho(x)$, $\rho(qx)$, and $\rho\bigl(q^2x\bigr)$ yields the relations
\begin{gather}
{A(qx) \over A\bigl(q^2x\bigr)} = -{\rho(qx)(1-qx)\theta\bigl({\rho(x) \over \rho(q^2x)}\bigr) \theta\bigl(a_1a_2x\rho(x)\rho(q^2x)\bigr) \over \rho\bigl(q^2x\bigr)^2 ((a_1+a_2)qx-b_1-b_2)\theta\bigl({\rho(x) \over \rho(qx)}\bigr) \theta(a_1a_2x\rho(x)\rho(qx))}, \label{eqn:A_rho_relation_1} \\
{A\bigl(q^2x\bigr) \over A(x)} = {\rho\bigl(q^2x\bigr)^2 (a_1a_2qx-b_1b_2) \theta\bigl({\rho(qx) \over \rho(x)}\bigr) \theta(a_1a_2x\rho(x)\rho(qx)) \over (1-qx)\theta\bigl({\rho(qx) \over \rho(q^2x)}\bigr) \theta\bigl(a_1a_2x\rho(qx)\rho\bigl(q^2x\bigr)\bigr)}, \label{eqn:A_rho_relation_2} \\
{A(x) \over A(qx)} = {((a_1+a_2)qx-b_1-b_2) \theta\bigl({\rho(q^2x) \over \rho(qx)}\bigr) \theta\bigl(a_1a_2x\rho(qx)\rho\bigl(q^2x\bigr)\bigr) \over \rho(qx) (a_1a_2qx-b_1b_2)\theta\bigl({\rho(q^2x) \over \rho(x)}\bigr) \theta\bigl(a_1a_2x\rho(x)\rho\bigl(q^2x\bigr)\bigr)}. \label{eqn:A_rho_relation_3}
\end{gather}
There are really only two distinct relations here, since~\eqref{eqn:A_rho_relation_3} may be obtained by multiplying together~\eqref{eqn:A_rho_relation_1} and~\eqref{eqn:A_rho_relation_2}. These equations may at once be used to determine a functional equation satisfied by the function $\rho$: Changing $x \mapsto qx$ in~\eqref{eqn:A_rho_relation_3} yields an expression for $A(qx)/A\bigl(q^2x\bigr)$ in terms of $\rho$, which may then be compared with the expression for this ratio given already by~\eqref{eqn:A_rho_relation_1}. The result of this is the formula~\eqref{eqn:rho_functional_equation}.

The question of how to express $A$ in terms of $\rho$ is rather more involved. The approach which seems most natural is to rewrite~\eqref{eqn:A_rho_relation_2} in the form which asserts that
\begin{equation} \label{eqn:constant_expression_involving_A_and_rho}
{A(x) A(qx) \theta\bigl({\rho(x) \over \rho(qx)}\bigr) \theta (a_1a_2x\rho(x)\rho(qx) ) \over x\rho(x)\rho(qx) \bigl(qx,{b_1b_2 \over a_1a_2x}\bigr)_\infty}
\end{equation}
is invariant under $x \mapsto qx$. If this expression is assumed to be analytic as a function of $x$ throughout $\C \backslash \{0\}$, then we may conclude that it must be constant, and from the fact that
\[ {}_2^{}\psi_2^\ast(1,y) = {\bigl({b_1b_2 \over a_1a_2}\bigr)_\infty \over (q)_\infty} \theta(a_1y)\theta(b_1y), \]
its value must be
\begin{equation} \label{eqn:conjectural_value}
{A(1) A(q) \theta\bigl({\rho(1) \over \rho(q)}\bigr) \theta(a_1a_2\rho(1)\rho(q)) \over \rho(1)\rho(q) \bigl(q,{b_1b_2 \over a_1a_2}\bigr)_\infty} = -{a_1a_2 \over (q)_\infty^2} {}_2^{}\psi_2^\ast(q,1/a_1) = -a_1a_2\left({b_1 \over a_1},{b_1 \over a_2},{b_2 \over a_1},{b_2 \over a_2}\right)_{ \infty},
\end{equation}
the $q$-Gauss identity~\eqref{eqn:q-Gauss} having been used in the final equality here. The resulting equation expresses $A(x)A(qx)$ in terms of $\rho$. Multiplying it by~\eqref{eqn:A_rho_relation_3} then yields the desired formula~\eqref{eqn:2psi2_A_rho_relation} for $A$ in terms of $\rho$. There is, of course, a problem with this argument: It is not at all obvious that~\eqref{eqn:constant_expression_involving_A_and_rho} should be analytic throughout $\C \backslash \{0\}$. As has been alluded to previously, the functions~$A$ and $\rho$ are multivalued; the function $\rho$ has branch points as described at the end of the present proof. There is also the question of whether~\eqref{eqn:constant_expression_involving_A_and_rho} might have poles arising from the zeros of the $q$-Pochhammer symbols on its denominator. It turns out, however, that~\eqref{eqn:constant_expression_involving_A_and_rho} is indeed identically equal to the right-hand side of~\eqref{eqn:conjectural_value}, but this will now be established by a different argument which makes no assumptions regarding the analytic properties of~\eqref{eqn:constant_expression_involving_A_and_rho}. Instead, the following argument is based on an analysis of a $q$-Wronskian determinant.

An equivalent formulation of~\eqref{eqn:2psi2ast_functional_eqn_wrt_x}, obtained by multiplying both sides by $\theta(x/y)$, is the statement that $\theta(x/y) {}_2^{}\psi_2^\ast(x,y)$ is a solution of the $q$-difference equation
\[ \left(1-{b_1b_2 \over a_1a_2qx}\right) f(x) - x\left( (a_1+a_2)x - {b_1+b_2 \over q} \right) f(qx) - a_1a_2qx^3 (1-qx) f\bigl(q^2x\bigr) =0 \]
for any value of $y$. This solution is analytic for all non-zero values of $x$ and $y$. By Lemma~\ref{lem:Wronskian}, the $q$-Wronskian determinant formed from the two solutions $f_1(x)=\theta(x/y) {}_2^{}\psi_2^\ast(x,y)$ and $f_2(x)=\theta(x/z) {}_2^{}\psi_2^\ast(x,z)$, namely
\begin{align}
\mathcal{W}(x,y,z) &= \theta\left( {qx \over y} \right) {}_2^{}\psi_2^\ast(qx,y) \cdot \theta\left( {x \over z} \right) {}_2^{}\psi_2^\ast(x,z) - \theta\left( {x \over y} \right) {}_2^{}\psi_2^\ast(x,y) \cdot \theta\left( {qx \over z} \right) {}_2^{}\psi_2^\ast(qx,z) \nonumber \\
&= {1 \over x}\theta\left( {x \over y} \right) \theta\left( {x \over z} \right) ( z \sp {}_2^{}\psi_2^\ast(x,y) {}_2^{}\psi_2^\ast(qx,z) - y \sp {}_2^{}\psi_2^\ast(qx,y){}_2^{}\psi_2^\ast(x,z) ), \label{eqn:2psi2_Wronskian_def}
\end{align}
satisfies the relation
\[ \mathcal{W}(qx,y,z) = -{1-{b_1b_2 \over a_1a_2qx} \over a_1a_2qx^3(1-qx)} \mathcal{W}(x,y,z). \]
Since, regarded as a function of $x$, $\mathcal{W}$ is analytic on $\C \backslash \{0\}$, it follows from Lemma~\ref{lem:homogeneous_first-order_q-difference_solution} that
\[ \mathcal{W}(x,y,z) = g(x,y,z) \left(qx,{b_1b_2 \over a_1a_2x}\right)_{ \infty} \]
for some function $g$ which, as a function of $x$, belongs to the space $\Theta_3(a_1a_2q)$. Since $\mathcal{W}(x,y,z)$ has zeros at $x=y$ and at $x=z$, it follows from Lemmas~\ref{lem:theta_root} and~\ref{lem:theta_factorization} that
\begin{equation} \label{eqn:initial_2psi2_Wronskian_factorization}
g(x,y,z) = h_1(y,z) \theta\left( {x \over y} \right) \theta\left( {x \over z} \right) \theta(a_1a_2qxyz)
\end{equation}
for some function $h_1$ which is independent of $x$.

Regarded as functions of $y$, both ${}_2^{}\psi_2^\ast(x,y)$ and $y {}_2^{}\psi_2^\ast(qx,y)$ belong to the space $\Theta_2(a_1a_2x)$. By Lemma~\ref{lem:theta_determinant_lemma}, it follows that
\[ \left| \begin{matrix} {}_2^{}\psi_2^\ast(x,y) &{}_2^{}\psi_2^\ast(x,z) \\
y \sp {}_2^{}\psi_2^\ast(qx,y) &z \sp {}_2^{}\psi_2^\ast(qx,z) \end{matrix} \right| = h_2(x) \Delta_2(y,z) \theta(a_1a_2xyz) \]
for some function $h_2$ which is independent of $y$ and $z$. Consequently
\[ \mathcal{W}(x,y,z) = {z h_2(x) \over x}\theta\left( {x \over y} \right) \theta\left( {x \over z} \right) \theta\left({y \over z} \right) \theta(a_1a_2xyz). \]
From comparing this with~\eqref{eqn:initial_2psi2_Wronskian_factorization}, it follows that
\begin{equation} \label{eqn:prelim_2psi2_Wronskian_factorization}
\mathcal{W}(x,y,z) = {Cz \over x} \left(qx,{b_1b_2 \over a_1a_2x}\right)_{ \infty} \theta\left({y \over z} \right) \theta\left( {x \over y} \right) \theta\left( {x \over z} \right) \theta(a_1a_2xyz)
\end{equation}
for some constant $C$ to be determined. In order to find the value of $C$, let us take $y=q/b_1$ and~${z=q/b_2}$. Then we have from the defining formula~\eqref{eqn:2psi2_Wronskian_def}
\begin{align*}
\mathcal{W}\left(x,{q \over b_1},{q \over b_2}\right)= {}&{q \over x}\theta\left( {b_1x \over q} \right) \theta\left( {b_2x \over q} \right) \left( {1 \over b_2} {}_2^{}\psi_2^\ast\left(x,{q \over b_1}\right) {}_2^{}\psi_2^\ast\left(qx,{q \over b_2}\right)\right.\\
 &\left.- {1 \over b_1} {}_2^{}\psi_2^\ast\left(qx,{q \over b_1}\right){}_2^{}\psi_2^\ast\left(x,{q \over b_2}\right) \right) \\
={}&{qb_1 \over x(b_1-b_2)} \left(x,qx,{b_1b_2 \over a_1a_2x},{b_1b_2 \over a_1a_2qx},{b_1 \over a_1},{b_1 \over a_2},{b_2 \over a_1},{b_2 \over a_2},q\right)_{ \infty} \\
&\times\theta\left( {b_2 \over b_1} \right) \theta\left( {b_1x \over q} \right) \theta\left( {b_2x \over q} \right) \\
& \times \left( {1 \over b_2} {}_2\psi_2\left(x,{q \over b_1}\right) {}_2\psi_2\left(qx,{q \over b_2}\right) - {1 \over b_1} {}_2\psi_2\left(qx,{q \over b_1}\right){}_2\psi_2\left(x,{q \over b_2}\right) \right)
\end{align*}
and from the factorization~\eqref{eqn:prelim_2psi2_Wronskian_factorization}
\[ \mathcal{W}\left(x,{q \over b_1},{q \over b_2}\right) = {Cq \over b_2x} \left(qx,{b_1b_2 \over a_1a_2x}\right)_{ \infty} \theta\left({b_2 \over b_1} \right) \theta\left( {b_1x \over q} \right) \theta\left( {b_2x \over q} \right) \theta\left({a_1a_2q^2x \over b_1b_2}\right). \]
From equating these two expressions for $\mathcal{W}(x,q/b_1,q/b_2)$, it follows that the constant $C$ is given~by
\begin{align*}
 C = {}&{b_1b_2 \bigl(x,{b_1 \over a_1},{b_1 \over a_2},{b_2 \over a_1},{b_2 \over a_2}\bigr)_\infty \over (b_1-b_2) \bigl({a_1a_2q^2x \over b_1b_2}\bigr)_\infty}\\
 &\times\left( {1 \over b_2} {}_2\psi_2\left(x,{q \over b_1}\right) {}_2\psi_2\left(qx,{q \over b_2}\right) - {1 \over b_1} {}_2\psi_2\left(qx,{q \over b_1}\right){}_2\psi_2\left(x,{q \over b_2}\right) \right).
 \end{align*}
In this formula the value of $x$ is arbitrary. The simplest choice is to take $x=0$ (the right-hand side being analytic there); this yields the value
\[ C = \left({b_1 \over a_1},{b_1 \over a_2},{b_2 \over a_1},{b_2 \over a_2}\right)_{ \infty}. \]
Hence
\[ \mathcal{W}(x,y,z) = {z \over x} \left({b_1 \over a_1},{b_1 \over a_2},{b_2 \over a_1},{b_2 \over a_2},qx,{b_1b_2 \over a_1a_2x}\right)_{ \infty} \theta\left({y \over z} \right) \theta\left( {x \over y} \right) \theta\left( {x \over z} \right) \theta(a_1a_2xyz) \]
and so
\begin{gather}
z \sp {}_2^{}\psi_2^\ast(x,y) {}_2^{}\psi_2^\ast(qx,z) - y \sp {}_2^{}\psi_2^\ast(qx,y){}_2^{}\psi_2^\ast(x,z)\nonumber\\
\qquad = z \left({b_1 \over a_1},{b_1 \over a_2},{b_2 \over a_1},{b_2 \over a_2},qx,{b_1b_2 \over a_1a_2x}\right)_{ \infty} \theta\left({y \over z} \right) \theta(a_1a_2xyz). \label{eqn:2psi2_Wronskian_factorization}
\end{gather}
In terms of the functions $A$ and $\rho$, the left-hand side of this equation is
\begin{gather*}
A(x) A(qx) \left( z \theta\left({y \over \rho(x)}\right) \theta (a_1a_2xy\rho(x) ) \theta\left({z \over \rho(qx)}\right) \theta (a_1a_2qxz\rho(qx) )\right. \\
\left.\qquad{}- y \theta\left({y \over \rho(qx)}\right) \theta (a_1a_2qxy\rho(qx) ) \theta\left({z \over \rho(x)}\right) \theta (a_1a_2xz\rho(x) ) \right).
\end{gather*}
Weierstrass's three-term identity (for which see~\cite[Exercise~2.16\,(i)]{Gasper2004}) may be used to equate this to
\[ -{z A(x) A(qx) \over a_1a_2x \rho(x)\rho(qx)} \theta\left({\rho(x) \over \rho(qx)}\right) \theta (a_1a_2x\rho(x)\rho(qx) ) \theta\left({y \over z} \right) \theta(a_1a_2xyz). \]
The identity~\eqref{eqn:2psi2_Wronskian_factorization} may therefore be expressed in the equivalent form
\begin{equation} \label{eqn:conjectural_value_established}
A(x) A(qx) = -{a_1a_2x\rho(x)\rho(qx)\bigl({b_1 \over a_1},{b_1 \over a_2},{b_2 \over a_1},{b_2 \over a_2},qx,{b_1b_2 \over a_1a_2x}\bigr)_{ \infty} \over \theta\bigl({\rho(x) \over \rho(qx)}\bigr) \theta(a_1a_2x\rho(x)\rho(qx))},
\end{equation}
which is the equality between~\eqref{eqn:constant_expression_involving_A_and_rho} and the right-hand side of~\eqref{eqn:conjectural_value} claimed earlier. Alternatively, this equation may be obtained from~\eqref{eqn:2psi2_Wronskian_factorization} directly, without the use of Weierstrass's three-term identity, by setting $z=\rho(x)$. Multiplying together~\eqref{eqn:A_rho_relation_3} and~\eqref{eqn:conjectural_value_established} yields{\samepage
\begin{align*}
A(x)^2 ={}& -{a_1a_2x \rho(x) ((a_1+a_2)qx-b_1-b_2 ) \bigl(qx,{b_1b_2 \over a_1a_2x},{b_1 \over a_1},{b_1 \over a_2},{b_2 \over a_1},{b_2 \over a_2}\bigr)_\infty \theta\bigl( {\rho(q^2x) \over \rho(qx)} \bigr) \over (a_1a_2qx-b_1b_2) \theta\bigl({\rho(x) \over \rho(qx)}\bigr) \theta\bigl({\rho(q^2x) \over \rho(x)}\bigr) \theta\bigl(a_1a_2x\rho(x)\rho(qx) ) \theta \bigl(a_1a_2x\rho(x)\rho\bigl(q^2x\bigr) \bigr)} \\
& \times \theta\bigl( a_1a_2x\rho(qx)\rho\bigl(q^2x\bigr)\bigr).
\end{align*}
An application of the relation~\eqref{eqn:rho_functional_equation}, with $x \mapsto x/q$, shows that this is equivalent to~\eqref{eqn:2psi2_A_rho_relation}.}

To obtain the formula for $\rho$ given in~\eqref{eqn:rho_integral_formula}, take $y=\pm1/\sqrt{a_1a_2x}$ in~\eqref{eqn:weak_2psi2ast_factorization} and divide the resulting equations to obtain
\[ {_2^{}\psi_2^\ast(x,1/\sqrt{a_1a_2x}) \over _2^{}\psi_2^\ast(x,-1/\sqrt{a_1a_2x})} = -{\theta \bigl( \rho(x)\sqrt{a_1a_2x} \bigr)^2 \over \theta \bigl({-}\rho(x)\sqrt{a_1a_2x} \bigr)^2}. \]
From here,~\eqref{eqn:rho_integral_formula} follows at once from one of Jacobi's inversion formulae, for which see~\cite[p.~480]{Whittaker1927}, viz.\footnote{In~\cite{Whittaker1927}, this classical result is stated in terms of the Jacobian elliptic function $\operatorname{sn}$. Its formulation given here may be obtained using~\cite[equation~(A), p.~492]{Whittaker1927}.}
\begin{gather*}
 y = {\theta(x)^2 \over \theta(-x)^2} \quad\; \text{if and only if} \\
 x = \exp \Bigg( {1 \over \theta(\sqrt{q})\theta(-\sqrt{q})} \int_0^y {\dd u \over \sqrt{u\Bigl( 1-{\theta(\sqrt{q})^2 \over \theta(-\sqrt{q})^2}u\Bigr)\Bigl( 1-{\theta(-\sqrt{q})^2 \over \theta(\sqrt{q})^2}u\Bigr)}} \Bigg).
 \end{gather*}
(The integral should be understood here in its multi-valued sense, so that the complete set of solutions of $y=\theta(x)^2/\theta(-x)^2$ for $x$ in terms of $y$ is obtained by taking the complete set of branches of the integral. If $x=x_0$ is one solution, then the complete set of solutions is given by~${x=q^n x_0}$ and $x=q^n/x_0$.)

The final part of Theorem~\ref{thm:2psi2_factorization}, pertaining to the branch points of the function $\rho$, remains to be proved. The right-hand side of~\eqref{eqn:rho_integral_formula} is, as a function of $x$, analytic at every point of~${\C \backslash \{0\}}$ except those for which the upper limit of integration coincides with a branch point of the integrand. The integrand has four branch points: $0$, $\infty$, and the two points \smash{$\theta (\sqrt{q} )^2/\theta(-\sqrt{q})^2$} and \smash{$\theta(-\sqrt{q})^2/\theta(\sqrt{q})^2$}. Hence, $\rho$ is analytic except for branch points which occur exactly at the values of $x$ which satisfy one of the following four equations:
\begin{gather}
_2^{}\psi_2^\ast\bigl(x,1/\sqrt{a_1a_2x}\bigr) = 0, \qquad
_2^{}\psi_2^\ast\bigl(x,-1/\sqrt{a_1a_2x}\bigr) = 0, \nonumber \\
\theta(-\sqrt{q})^2 {}_2^{}\psi_2^\ast\bigl(x,1/\sqrt{a_1a_2x}\bigr) + \theta\bigl(\sqrt{q}\bigr)^2 {}_2^{}\psi_2^\ast\bigl(x,-1/\sqrt{a_1a_2x}\bigr) = 0, \label{eqn:branch_point_condition_3} \\
\theta\bigl(\sqrt{q}\bigr)^2 {}_2^{}\psi_2^\ast\bigl(x,1/\sqrt{a_1a_2x}\bigr) + \theta\bigl(-\sqrt{q}\bigr)^2 {}_2^{}\psi_2^\ast\bigl(x,-1/\sqrt{a_1a_2x}\bigr) = 0. \label{eqn:branch_point_condition_4}
\end{gather}
The left-hand sides of~\eqref{eqn:branch_point_condition_3} and~\eqref{eqn:branch_point_condition_4} are equal to
\[ \theta(-1)^2{}_2^{}\psi_2^\ast\bigl(x,\sqrt{q/a_1a_2x}\bigr) \qquad \text{and} \qquad \theta(-1)^2{}_2^{}\psi_2^\ast\bigl(x,-\sqrt{q/a_1a_2x}\bigr), \]
respectively. This may be seen from taking \smash{$y=\pm\sqrt{q/a_1a_2x}$}, $z_1=1/\!\sqrt{a_1a_2x}$ and $z_2=\!-1/\!\sqrt{a_1a_2x}$ in the identity
\[ _2^{}\psi_2^\ast(x,y) = {z_2\theta\bigl({y \over z_2}\bigr) \theta(a_1a_2xyz_2) {}_2^{}\psi_2^\ast(x,z_1) - z_1\theta\bigl({y \over z_1}\bigr) \theta(a_1a_2xyz_1) {}_2^{}\psi_2^\ast(x,z_2) \over z_2\theta\bigl({z_1 \over z_2}\bigr) \theta(a_1a_2xz_1z_2)}, \]
which is simply an instance of~\eqref{eqn:Slater_general_linear_bilateral_relation} since, as a function of $y$, $_2^{}\psi_2^\ast(x,y)$ belongs to the space $\Theta_2(a_1a_2x)$.
\end{proof}

\begin{Remark}\label{rmk:Gustafson_equivalence}
It happens that~\eqref{eqn:2psi2_Wronskian_factorization} is equivalent to the special case $r=2$ of Gustafson's~$A_r$ generalization of the $_1\psi_1$ summation~\cite[Theorem~1.17]{Gustafson1987}. The general case of Gustafson's formula is also obtainable from an application of Lemmas~\ref{lem:theta_determinant_lemma} and~\ref{lem:homogeneous_first-order_q-difference_solution}. This is described in the next section.
\end{Remark}

\begin{Remark}
From setting $z=y\sqrt{q}$ in~\eqref{eqn:2psi2_Wronskian_factorization}, it follows that, as a function of the variable $y$, the function
\[ {}_2^{}\psi_2^\ast(x,y) {}_2^{}\psi_2^\ast(qx,y\sqrt{q}) + {\theta(\sqrt{q}) \over 2\sqrt{q}} \left(qx,{b_1b_2 \over a_1a_2x},{b_1 \over a_1},{b_1 \over a_2},{b_2 \over a_1},{b_2 \over a_2}\right)_{ \infty} \theta\bigl(a_1a_2xy^2\sqrt{q}\bigr) \]
belongs to the space $\Theta_2(a_1a_2x\sqrt{q};\sqrt{q})$. \big(Here the notation $\Theta_n(c;q)$ is used in place of $\Theta_n(c)$ to allow for a change of base; thus the space in question here consists of all functions analytic on~$\C \backslash \{0\}$ which satisfy the relation $f(y\sqrt{q}) = f(y)/\bigl(a_1a_2xy^2\sqrt{q}\bigr)$.\big)
\end{Remark}

\section[The general \_r psi\_r series]{The general $\boldsymbol{_r\psi_r}$ series}

The purpose of this section is to demonstrate, in outline, how the result of Theorem~\ref{thm:2psi2_factorization} may be generalized to describe the relationship between the undetermined functions in Theorem~\ref{thm:general_rpsir_factorization} for any positive integer $r$. It turns out that the function $A$ in the factorization~\eqref{eqn:general_rpsir_factorization} is always expressible in terms of the functions $\rho_j$ and the variables (excluding $y$). The method is based on that of the previous section.

Let $\lambda_j$ denote the coefficients in the expansion
\[ x\prod_{j=1}^r (1-a_jy) - \prod_{j=1}^r (1-b_j y/q)= \sum_{j=0}^r \lambda_j(x) y^j. \]
Then from the elementary identity
\[ \sum_{n=-\infty}^\infty {(a_1y,\dots,a_ry)_n \over (b_1y,\dots,b_ry)_n} \prod_{j=1}^r\bigl(1-b_jyq^{n-1}\bigr) x^n = \sum_{n=-\infty}^\infty {(a_1y,\dots,a_ry)_n \over (b_1y,\dots,b_ry)_n} \prod_{j=1}^r(1-a_jyq^n) x^{n+1}, \]
which is the obvious generalization of~\eqref{eqn:elementary_2psi2_shift}, it is apparent that
\[ \sum_{j=0}^r y^j \lambda_j(x) {}_r\psi_r\bigl(xq^j,y\bigr) = 0, \]
or equivalently,
\begin{equation} \label{eqn:rpsir_ast_linear_relation}
\sum_{j=0}^r {y^j \lambda_j(x) (x)_j \over \bigl({b_1\cdots b_r \over a_1\cdots a_r xq^j}\bigr)_j} {}_r^{}\psi_r^\ast\big(xq^j,y\big) = 0.
\end{equation}
Hence, for any given non-zero complex numbers $y_1, y_2, \dots, y_r$, for $k=1,2,\dots,r$, each of the functions $f_k(x) = \theta(x/y_k) {}_r^{}\psi_r^\ast(x,y_k)$ is a solution of the $q$-difference equation
\[ \sum_{j=0}^r {(-1)^j q^{j(j-1) \over 2} x^j \lambda_j(x) (x)_j \over \bigl({b_1\cdots b_r \over a_1\cdots a_r xq^j}\bigr)_j} f\bigl(xq^j\bigr) = 0. \]
By Lemma~\ref{lem:Wronskian}, the $q$-Wronskian
\[ \mathcal{W}(x,y_1,y_2,\dots,y_r) = \det_{1 \leq i,j \leq r}\bigl( \theta\bigl(q^{j-1}x/y_i\bigr) {}_r^{}\psi_r^\ast\bigl(q^{j-1}x,y_i\bigr) \bigr) \]
therefore satisfies the relation
\[ {\mathcal{W}(qx,y_1,y_2,\dots,y_r) \over \mathcal{W}(x,y_1,y_2,\dots,y_r)} = {\lambda_0(x) \bigl({b_1\cdots b_r \over a_1\cdots a_r xq^r}\bigr)_r \over q^{r(r-1) \over 2} x^r \lambda_r(x) (x)_r} = {(-1)^{r+1} q^{-{r(r-1) \over 2}} x^{-r-1} \bigl({b_1\cdots b_r \over a_1\cdots a_r xq^{r-1}}\bigr)_{r-1} \over a_1\cdots a_r (qx)_{r-1}}. \]
By Lemmas~\ref{lem:theta_root} and~\ref{lem:homogeneous_first-order_q-difference_solution}, this implies that
\[ {\mathcal{W}(x,y_1,y_2,\dots,y_r) \over \theta(x/y_1)\theta(x/y_2) \cdots \theta(x/y_r) \theta\bigl(a_1 y_1 \cdots a_r y_r q^{r(r-1) \over 2} x\bigr) \prod_{j=1}^{r-1} \bigl(xq^j,{b_1\cdots b_r \over a_1\cdots a_r x q^{j-1}}\bigr)_\infty} \]
is independent of $x$. Equivalently, the slightly modified determinant
\[ \widetilde{\mathcal{W}}(x,y_1,y_2,\dots,y_r) = \det_{1 \leq i,j \leq r}\bigl( y_i^{j-1} {}_r^{}\psi_r^\ast\bigl(q^{j-1}x,y_i\bigr) \bigr) \]
has the property that
\[ {\widetilde{\mathcal{W}}(x,y_1,y_2,\dots,y_r) \over \theta(a_1y_1\cdots a_ry_r x)\prod_{j=1}^{r-1} \bigl(xq^j,{b_1\cdots b_r \over a_1\cdots a_r x q^{j-1}}\bigr)_\infty} \]
is independent of $x$. On the other hand, Lemma~\ref{lem:theta_determinant_lemma} implies that
\[ {\widetilde{\mathcal{W}}(x,y_1,y_2,\dots,y_r) \over \Delta_r(y_1,y_2,\dots,y_r) \theta(a_1y_1\cdots a_ry_r x)} \]
is independent of $y_1,y_2,\dots,y_r$. Consequently
\begin{equation} \label{eqn:rpsir_Wronskian_ratio}
{\widetilde{\mathcal{W}}(x,y_1,y_2,\dots,y_r) \over \Delta_r(y_1,y_2,\dots,y_r) \theta(a_1y_1\cdots a_ry_r x)\prod_{j=1}^{r-1} \bigl(xq^j,{b_1\cdots b_r \over a_1\cdots a_r x q^{j-1}}\bigr)_\infty}
\end{equation}
is independent of $x$ and of $y_1,y_2,\dots,y_r$. Its value, which is
\begin{equation} \label{eqn:value_of_rpsir_Wronskian_ratio}
(q)_\infty^{-{(r-1)(r-2) \over 2}} \prod_{i,j=1}^r \left({b_i \over a_j}\right)_{ \infty},
\end{equation}
may be found by setting $y_j = q/b_j$ for each $j$ and then, after cancelling factors which are not analytic at the origin, setting $x=0$. Per Remark~\ref{rmk:Gustafson_equivalence}, the assertion that~\eqref{eqn:rpsir_Wronskian_ratio} and~\eqref{eqn:value_of_rpsir_Wronskian_ratio} are equal is equivalent to Gustafson's $A_r$ generalization of~\eqref{eqn:1psi1}. The relation between these two formulae may be seen by writing Gustafson's multi-series as a determinant of single series. (The summand of Gustafson's series is expressible as a Vandermonde determinant.) A rather different determinant of basic hypergeometric series, in which each entry is itself a multi-series, is considered by Aomoto and Ito~\cite[Theorem~3.9]{Aomoto2008}; their result generalizes Gustafson's $C_n$ generalization of~\eqref{eqn:6psi6}.

From the equality between~\eqref{eqn:rpsir_Wronskian_ratio} and~\eqref{eqn:value_of_rpsir_Wronskian_ratio}, there follows a relation between the functions $A$ and $\rho_j$ appearing in the factorization~\eqref{eqn:general_rpsir_factorization}. This expresses the product $A(x)A(qx) \dots A\bigl(q^{r-1}x\bigr)$ in terms of the $\rho_j$. On the other hand, substituting~\eqref{eqn:general_rpsir_factorization} into~\eqref{eqn:rpsir_ast_linear_relation} and dividing both sides by~$A(x)$ yields the equation
\[ \sum_{j=1}^r {y^j \lambda_j(x) (x)_j \over \bigl({b_1\cdots b_r \over a_1\cdots a_r xq^j}\bigr)_j} {A\bigl(xq^j\bigr) \over A(x)} \theta\bigg( {y \over \rho_1\bigl(xq^j\bigr)} \bigg) \cdots \theta\bigg( {y \over \rho_r\bigl(xq^j\bigr)} \bigg) = - \lambda_0(x) \theta\left( {y \over \rho_1(x)} \right) \cdots \theta\left( {y \over \rho_r(x)} \right). \]
Taking $r$ different values of $y$ in this equation yields a system of simultaneous equations for the ratios $A\bigl(xq^j\bigr)/A(x)$ for $j=1,2,\dots,r$. The values of these ratios obtained by solving the equations may then be combined with the expression already obtained for the product $A(x)A(qx) \cdots A\bigl(q^{r-1}x\bigr)$ to determine $A(x)$ in terms of the $\rho_j$. The details of this will appear in a later paper.

\subsection*{Acknowledgements}
My thanks are due to the referees for their comments and suggestions and for drawing my attention to the papers~\cite{Ito2014,Ito2008, Rosengren2006}. The author is supported by the FCT project UIDP/00208/2020.

\pdfbookmark[1]{References}{ref}
\LastPageEnding

\end{document}